\documentclass[letterpaper,11pt,reqno]{amsart}

\makeatletter
\usepackage{amssymb}
\usepackage{latexsym}
\usepackage{amsbsy}
\usepackage{amsfonts}
\usepackage{hyperref}
\usepackage{graphicx}
\usepackage{enumerate}
\usepackage{color}

\usepackage{tikz}

\def\marginpar#1{\ignorespaces}

\DeclareMathOperator\sgn{sgn}
\newtheorem{theorem}{Theorem}[section]
\newtheorem{lemma}[theorem]{Lemma}
\newtheorem{proposition}[theorem]{Proposition}
\newtheorem{corollary}[theorem]{Corollary}

\numberwithin{equation}{section}
\makeatother
\begin{document}
\title[Argmin process]{The argmin process of random walks and L\'evy processes}

\author[Jim Pitman]{{Jim} Pitman}
\address{Statistics department, University of California, Berkeley. Email: 
} \email{pitman@stat.berkeley.edu}

\author[Wenpin Tang]{{Wenpin} Tang}
\address{Statistics department, University of California, Berkeley. Email: 
} \email{wenpintang@stat.berkeley.edu}

\date{\today} 
\begin{abstract}
In this paper we consider the argmin process of random walks and L\'evy processes. We prove that they enjoy the Markov property, and provide their transition kernels in some special cases. 
\end{abstract}

\maketitle


\setcounter{tocdepth}{1}
\tableofcontents
\section{Introduction and main results}
\quad In recent work \cite{PTargmin}, we considered the {\em argmin process} $(\alpha_t; t \geq 0)$ of Brownian motion, defined by
\begin{equation}
\alpha_t : = \sup\left\{s \in [0,1]: B_{t+s}=\min_{u \in [0,1]}B_{t+u} \right\} \quad \mbox{for all}~t \geq 0.
\end{equation}
It is easy to see that the process $\alpha$ is stationary, with arcsine distributed invariant measure. We proved that $(\alpha_t; t \geq 0)$ is a Markov process with the Feller property, and computed its transition kernel $Q_t(x,\cdot)$ for $t>0$ and $x \in [0,1]$. 

\quad The purpose of this paper is to extend our previous results to random walks and L\'evy processes. Fix $N \geq 1$. We study the {\em argmin chain} $(A_N(n); n \geq 0)$ of a random walk $(S_n; n \geq 0)$, defined by
\begin{equation}
\label{argminchain}
A_N(n):=\sup \left\{1 \leq i \leq N; S_{n+i} = \min_{1\leq i \leq N} S_{n+i} \right\} \quad \mbox{for all}~n \geq 0,
\end{equation}
where $S_n:=\sum_{i=1}^n X_i$ is the $n^{th}$ partial sum of $(X_n; n \geq 0)$ (with convention $S_0:=0$), and $(X_n; n \geq 0)$ is a sequence of independent and identically distributed random variables with the cumulative distribution function $F$. This is  the discrete analog of the argmin process of Brownian motion.
A similar argument as in the Brownian case shows that $(A_N(n); n \geq 0)$ is a Markov chain. For $n \geq 1$, let
$$p_n: = \mathbb{P}(S_1 \geq 0, \cdots, S_n \geq 0) \quad \mbox{and} \quad \widetilde{p}_n: = \mathbb{P}(S_1 > 0, \cdots, S_n > 0).$$
Theorem \ref{GF} below recalls the classical theory of how the two sequences of probabilities $p_n$ and $\widetilde{p}_n$ are determined by the sequences of probabilities $\mathbb{P}(S_n \ge 0 )$ and $\mathbb{P}( S_n > 0 )$.
We give the transition matrix of the argmin chain $A_N$ in terms of $(p_n; n \geq 1)$ and $(\widetilde{p}_n; n \geq 1$), which can be made explicit for special choices of $F$. 

\begin{theorem}
\label{discreteth}
Whatever the common distribution $F$ of $(X_n; n \geq 0)$, the argmin chain $(A_N(n); n \geq 0)$ is a stationary and time-homogeneous Markov chain on $\{0,1, \ldots, N\}$. Let $\Pi_N(k)$, $k \in [0,N]$ be the stationary distribution, and $P_N(i,j)$, $i,j \in [0,N]$ be the transition probabilities of the argmin chain $(A_N(n); n \geq 0)$ on $[0,N]$. Then
\begin{equation}
\label{stationaryPi}
\Pi_N(k) = p_k \widetilde{p}_{N-k} \quad \mbox{for}~0 \leq k \leq N;
\end{equation}
\begin{equation}
\label{trans1}
P_N(i,N ) = 1 -\frac{\widetilde{p}_{N+1-i}}{\widetilde{p}_{N-i}} \quad \mbox{and} \quad P_N(i,i-1) = \frac{\widetilde{p}_{N+1-i}}{\widetilde{p}_{N-i}} \quad \mbox{for}~0<i \leq N;
\end{equation}
\begin{equation}
\label{trans2}
P_N(0,j ) = \frac{(p_j-p_{j+1}) \widetilde{p}_{N-j}}{\widetilde{p}_{N}} \quad \mbox{for}~0\leq  j < N \quad \mbox{and} \quad P_N(0,N) = 1 - \sum_{j=0}^{N-1} P_N(0,j).
\end{equation}
In particular,
\begin{enumerate}
\item 
If $(S_n; n \geq 0)$ is a random walk with continuous distribution and $\mathbb{P}(S_n>0) = \theta \in (0,1)$ for all $n \geq 1$. Let $(\theta)_{n \uparrow}: = \prod_{i=0}^{n-1} (\theta + i)$ be the Pochhammer symbol. Then
\begin{equation}
\label{disarcsine}
\Pi_{N}(k) = \frac{(\theta)_{k \uparrow} (\theta)_{N-k \uparrow}}{k! (N-k)!} \quad \mbox{for}~0 \leq k \leq N;
\end{equation}
\begin{equation}
\label{eq5354}
P_N(i, N) = \frac{1-\theta}{N+1-i} \quad \mbox{and} \quad P_N(i,i-1) = \frac{N+\theta-i}{N+1-i} \quad \mbox{for}~0<i \leq N;
\end{equation}
\begin{equation}
\label{eq55}
P_N(0, j) =  \frac{1 - \theta}{j+1} \binom{N}{j} \frac{(\theta)_{j \uparrow} (\theta)_{N-j \uparrow}}{(\theta)_{N \uparrow}} \quad \mbox{for}~0 \leq j < N;
\end{equation}
and
\begin{equation}
\label{eq56}
P_N(0,N) = \frac{2(1-\theta)}{N+1} - \frac{(1 - 2 \theta)(2 \theta)_{N \uparrow}}{(N+1)(\theta)_{N \uparrow}}.
\end{equation}
\item
If $(S_n; n \geq 0)$ is a simple symmetric random walk. Let $\lfloor x \rfloor$ be the integer part of $x$. Then
\begin{equation}
\label{eq513}
\Pi_N(k) =   \frac{\displaystyle \left(\frac{1}{2}\right)_{ \lfloor{\frac{k+1}{2}}\rfloor \uparrow} \left(\frac{1}{2}\right)_{\lfloor \frac{N-k}{2} \rfloor \uparrow} }{\displaystyle 2 \cdot \left \lfloor \frac{k+1}{2} \right \rfloor ! \left \lfloor \frac{N-k}{2}  \right \rfloor !}  \quad \mbox{for}~0 \leq k \leq N;
\end{equation}
For $0 < i \leq N$,
\begin{equation}
\label{eq514}
P_N(i,N) =     \left\{ \begin{array}{ccl}
         \frac{N-i}{N+1-i} & \mbox{if}~ N-i ~\mbox{is odd}; \\ [8 pt]
         1 & \mbox{if} ~N-i ~\mbox{is even};
                \end{array}\right.
\mbox{and} ~
P_N(i, i-1) = 1 - P_N(i,N);
\end{equation}
for $0 \leq j <N$,
\begin{equation}
\label{eq515}
P_N(0,j) = \left\{ \begin{array}{ccl}
          0 & \mbox{if}~ j ~\mbox{is odd}; \\ [8 pt]
         \frac{\displaystyle \binom{j}{\frac{j}{2}} \binom{2 \lfloor \frac{N}{2} \rfloor-j}{\lfloor \frac{N}{2} \rfloor-\frac{j}{2}}}{\displaystyle(j+2)\binom{2 \lfloor \frac{N}{2} \rfloor}{ \lfloor\frac{N}{2} \rfloor}} & \mbox{if} ~j ~\mbox{is even};
                \end{array}\right.
\end{equation}
and
\begin{equation}
\label{eq516}
P_N(0,N) = \left\{ \begin{array}{ccl}
         \frac{1}{N+1} & \mbox{if}~ N ~\mbox{is odd}; \\ [8 pt]
         \frac{2}{N+2} & \mbox{if} ~N ~\mbox{is even}.
                \end{array}\right.
\end{equation}
\end{enumerate}
\end{theorem}

\quad For the argmin chain $A_N$, the transition probability from $0$ to $N$ is given by \eqref{trans2} in the general case. But this probability is simplified to \eqref{eq56} and \eqref{eq516} in the two special cases. These identities are proved analytically by Lemmas \ref{lem11} and \ref{lem12}. We do not have a simple explanation, and leave combinatorial interpretations for interested readers.

\quad Let $(X_t; t \geq 0)$ be a real-valued L\'evy process. We consider the argmin process $(\alpha_t^X; t \geq 0)$ of $X$, defined by
\begin{equation}
\alpha_t^X : = \sup\left\{s \in [0,1]: X_{t+s}=\inf_{u \in [0,1]}X_{t+u} \right\} \quad \mbox{for all}~t \geq 0.
\end{equation}
We are particularly interested in the case where $X$ is a stable L\'evy process. We follow the notations in Bertoin \cite[Chapter VIII]{Bertoin}. Up to a multiple factor, a stable L\'evy process $X$ is entirely determined by a {\em scaling parameter} $\alpha \in (0.2]$, and a {\em skewness parameter} $\beta \in [-1,1]$. The characteristics exponent of a stable L\'evy process $X$ with parameters $(\alpha,\beta)$ is given by
\begin{equation*}
    \Psi(\lambda) : = \left\{ \begin{array}{ccl}        
  |\lambda|^{\alpha} (1 - i \beta \sgn(\lambda) \tan(\pi \alpha/2))  & \mbox{for}~\alpha \neq 1, \\ [8pt]
  |\lambda| (1 + i \frac{2 \beta}{\pi} \sgn(\lambda) \log|\lambda|)  & \mbox{for}~\alpha = 1.
\end{array}\right.
\end{equation*}
where $\sgn$ is the sign function. Let $\rho : = \mathbb{P}(X_1>0)$ be the {\em positivity parameter}. Zolotarev \cite[Section 2.6]{Zolo} found that 
\begin{equation}
\label{positivity}
\rho = \frac{1}{2} + (\pi \alpha)^{-1}\arctan(\beta \tan(\pi \alpha/2)) \quad \mbox{for}~\alpha \in (0,2].
\end{equation}
\quad If $X$ (resp. $-X$ ) is a subordinator, then almost surely $\alpha^X_t =0$ (resp. $\alpha^X_t = 1$) for all $t \geq 0$. Relying on the excursion theory, we generalize Pitman and Tang \cite[Theorem 1.2]{PTargmin}.

\begin{theorem}
\label{mainLevy}
~
\begin{enumerate}
\item
Let $(X_t; t \geq 0)$ be a L\'evy process. Then the argmin process $(\alpha^X_t; t \geq 0)$ of $X$ is a stationary and time-homogeneous Markov process.
\item
Let $(X_t; t \geq 0)$ be a stable L\'evy process with parameters $(\alpha,\beta)$, and assume that neither $X$ nor $-X$ is a subordinator. 
Let $\rho$ be defined by \eqref{positivity}. Then the argmin process $(\alpha^X_t; t \geq 0)$ of $X$ is a stationary Markov process, with generalized arcsine distributed invariant measure whose density is
\begin{equation}
f(x) : = \frac{\sin \pi \rho}{\pi} x^{-\rho} (1-x)^{\rho-1} \quad \mbox{for}~0<x<1.
\end{equation}
and Feller transition semigroup $Q^X_t(x, \cdot)$, $t >0$ and $x \in [0,1]$ where
\begin{equation}
\label{Qtrans}
Q^X_t(x,dy) =  \left\{ \begin{array}{ccl}        
 1_{\{0<y<1\}} \frac{\sin \pi \rho}{\pi} y^{-\rho} (1-y)^{\rho-1} dy  & \mbox{for}~0 \leq x \leq 1 <t, \\ [8pt]
 \left(\frac{1-x}{1-x+t}\right)^{1-\rho} \delta_{x-t}(dy) +  \frac{\sin \pi \rho}{\pi} \cdot \frac{ (1-y)^{\rho-1} (y+t-1)_{+}^{\rho} }{(y+t-x)}dy & \mbox{for}~0< t \leq x \leq 1\\ [8pt]
 \frac{\sin \pi \rho}{\pi (y+t-x)} y^{-\rho} (1-y)^{\rho-1} [(t-x)^{\rho}(1-x)^{1-\rho} + y^{\rho}(y+t-1)_{+}^{1-\rho}] dy  & \mbox{for}~0 \leq x<t \leq 1.
\end{array}\right.
\end{equation} 
\end{enumerate}
\end{theorem}
\vskip 12 pt
{\bf Organization of the paper:} The layout of the paper is as follows. 
\begin{itemize}
\item
In Section \ref{s5}, we study the argmin chain $(A_N(n); n \geq 0)$ for random walks. There Theorem \ref{discreteth} is proved.
\item
In Section \ref{s3}, we consider the argmin process $(\alpha^X_t; t \geq 0)$ of L\'evy process, and prove Theorem \ref{mainLevy}.
\end{itemize}
\section{The argmin chain of random walks}
\label{s5}
\quad In this section, we prove Theorem \ref{discreteth}. Recall the definition of the argmin chain $(A_N(n); n \geq 0)$ from \eqref{argminchain}. Fix $N  \geq 1$. Let $(\overset{\rightarrow}{X}_N(n); n \geq 0)$ be the moving window process of length $N$, defined by
$$\overset{\rightarrow}{X}_N(n):=(X_{n+1}, \ldots, X_{n+N}) \quad \mbox{for}~n \geq 0,$$
with associated partial sums $\overset{\rightarrow}{S^X_N}(n):=(0, X_{n+1}, X_{n+1}+X_{n+2}, \ldots, \sum_{i=1}^N X_{n+i})$. Similarly, let $(\overset{\leftarrow}{X}_N(n); n \geq 0)$ be the reversed moving window process of length $N$, defined by
$$\overset{\leftarrow}{X}_N(n):=(-X_{n}, \ldots, -X_{n-N+1}) \quad \mbox{for}~n \geq N,$$
with associated partial sums $\overset{\leftarrow}{S^X_N}(n):=(0, -X_{n}, -X_{n}-X_{n-1}, \ldots, -\sum_{i=1}^N X_{n+1-i})$. 
Note that $n+A_N(n)$ is the last time at which the minimum of $(S_k; k \geq 0)$ on $[n+1,n+N]$ is attained. So $A_N(n)$ is a function of $\overset{\rightarrow}{S^X_N}(n)$ or $\overset{\leftarrow}{S^X_N}(n+N)$. 
The following path decomposition is due to Denisov.

\begin{theorem} [Denisov's decomposition for random walks] \cite{Denisov}
\label{DenisovRW}
Let $S_n:= \sum_{i=1}^n X_i$, where $X_i$ are independent random variables. For $N \geq 1$, let 
$$A_N:=\sup\left\{0 \leq i \leq N: S_i=\min_{1 \leq k \leq N} S_k\right\}$$
be the last time at which $(S_k; k \geq 0)$ attains its minimum on $[0,N]$. For each positive integer $a$ with $0 \le a \le N$, given the event $\{A_N = a\}$, 
the random walk is decomposed into two conditionally independent pieces:
\begin{enumerate}[(a).]
\item
$(S_{a-k}-S_{a}; 0 \leq k \leq a)$ has the same distribution as $\overset{\leftarrow \quad}{S^X_{a}}(a)$ conditioned to stay non-negative;
\item
$(S_{a+k}-S_{a}; 0 \leq k \leq N - a)$ has the same distribution as $\overset{\rightarrow \quad \quad}{S^X_{N - a}} (a)$ conditioned to stay positive.
\end{enumerate}
\end{theorem}

\quad By Denisov's decomposition for random walks, it is easy to adapt the argument of Pitman and Tang \cite[Proposition 3.4]{PTargmin} to show that $(A_N(n); n \geq 0)$ is a time-homogeneous Markov chain on $\{0,1,\cdots, N\}$. Here the detail is omitted. 

\quad Now we compute the invariant distribution $\Pi_N$, and the transition matrix $P_N$ of  the argmin chain $(A_N(n); n \geq 0)$ on $\{0,1, \ldots, N\}$. To proceed further, we need the following result regarding the law of ladder epochs, originally due to Sparre Anderson \cite{SA}, Spitzer \cite{Spitzer} and Baxter \cite{Baxter}. It can be read from Feller \cite[Chapter XII$.7$]{Fellervol2}.
\begin{theorem} \cite{SA, Fellervol2}
\label{GF}
\begin{enumerate}
\item
Let $\tau_n:=\mathbb{P}(S_1 \geq 0, \ldots, S_{n-1} \geq 0, S_n<0)$ and $\tau(s): = \sum_{n=0}^{\infty} \tau_n s^n$. Then for $|s|<1$,
$$\log \frac{1}{1-\tau(s)} = \sum_{n=1}^{\infty} \frac{s^n}{n} \mathbb{P}(S_n<0).$$
\item 
Let $p_n:= \mathbb{P}(S_1 \geq 0, \ldots, S_n \geq 0)$ and $p(s): = \sum_{n=0}^{\infty} p_n s^n$. Then for $|s|<1$,
$$p(s) = \exp\left(\sum_{n=1}^{\infty} \frac{s^n}{n} \mathbb{P}(S_n \geq 0)\right).$$
\item
Let $\widetilde{p}_n:= \mathbb{P}(S_1 > 0, \ldots, S_n > 0)$ and $\widetilde{p}(s): = \sum_{n=0}^{\infty} \widetilde{p}_n s^n$. Then for $|s|<1$,
$$\widetilde{p}(s) = \exp\left(\sum_{n=1}^{\infty} \frac{s^n}{n} \mathbb{P}(S_n > 0)\right).$$
\end{enumerate}
\end{theorem}

\quad In the sequel, let $T_{-}:=\inf\{n \geq 1; S_n<0\}$ and $\widetilde{T}_{-}:=\inf\{n \geq 1; S_n \leq 0\}$ so that $p_n = \mathbb{P}(T_{-} > n)$ and $\widetilde{p}_n = \mathbb{P}(\widetilde{T}_{-} > n)$. 

\begin{proof}[Proof of Theorem \ref{discreteth}]
Observe that the distribution of the argmin of sums on $\{0,1,\cdots,N\}$ is the stationary distribution of the argmin chain. Following Feller \cite[Chapter XII.8]{Feller}, this is the discrete arcsine law
\begin{equation*}
\Pi_N(k) = p_k \widetilde{p}_{N-k} \quad \mbox{for}~0 \leq k \leq N.
\end{equation*}
Let $t_i:=\mathbb{P}(T_{-}=i) = p_{i-1} - p_i$ and $\widetilde{t}_i:=\mathbb{P}(\widetilde{T}_{-}=i) = \widetilde{p}_{i-1} - \widetilde{p}_i$ for $i >0$. Now we calculate the transition probabilities of the argmin chain. We distinguish two cases.
\vskip 6 pt
{\bf Case $1$.} The argmin chain starts at $0 < i \leq N$: $A_N(0) = i$. This implies that for all $k \in [1, i-1]$, $S_k \geq S_i$, and for all $k \in [i+1, N]$, $S_k > S_i$.
\begin{itemize}
\item 
If $S_{N+1}>S_i$, then the last time at which $(S_n)_{1 \leq n \leq N+1}$ attains its minimum is $i$, meaning that $A_{N}(1) = i-1$.
\item
If $S_{N+1} = S_i$, the the last time at which $(S_n)_{1 \leq n \leq N+1}$ attains its minimum is $N+1$, meaning that $A_N(1) = N$.
\end{itemize}
If we look forward from time $i$, $N+1$ is the first time at which the chain enters $(-\infty, 0]$. Consequently, for $0 < i \leq N$,
\begin{equation}
P_N(i,N) = \frac{\widetilde{t}_{N+1-i}}{\widetilde{p}_{N-i}} \quad \mbox{and} \quad P_N(i,i-1) = 1 - P_N(i,N),
\end{equation}
which leads to \eqref{trans1}.
\vskip 6 pt
{\bf Case $2$.} The argmin chain starts at $i=0$: $A_N(0) = 0$. For $0 \leq j <N$, let $j+1$ be the last time at which the minimum on $[1,N]$ is attained. 
\begin{itemize}
\item 
If $S_{N+1}>S_{j+1}$, then the last time at which $(S_n)_{1 \leq n \leq N+1}$ attains its minimum is $j+1$, meaning that $A_N(1)=j$.
\item
If $S_{N+1} = S_{j+1}$, then the last time at which $(S_n)_{1 \leq n \leq N+1}$ attains its minimum is $N+1$, meaning that $A_N(1)=N$.
\end{itemize}
If we look backward from time $j+1$, the origin is the first time at which the reversed walk enters $(-\infty, 0)$. So for $0 \leq j <N$,
\begin{equation}
P_{N}(0,j) = \frac{t_{j+1} \widetilde{p}_{N-j}}{\widetilde{p}_N},
\end{equation}
which yields \eqref{trans2}. The above formula fails for $j=N$, but $P_N(0, N) = 1 -\sum_{j=0}^{N-1}P_N(0,j)$. 
\end{proof}

\quad We complete the proof of Theorem \ref{discreteth} in the following two subsections. In Section \ref{s51}, we consider the non-lattice case, and in Section \ref{s52}, we deal with the lattice case.
\subsection{$F$ is continuous and $\mathbb{P}(S_n>0) = \theta \in (0,1)$}
\label{s51}
From Theorem \ref{GF}, we deduce the well known facts that
\begin{equation*}
\log p(s) = \theta \sum_{n=1}^{\infty} \frac{s^n}{n} \Longrightarrow p(s) = (1-s)^{-\theta} = 1+ \sum_{n=1}^{\infty} \frac{(\theta)_{n \uparrow}}{n!} s^n,
\end{equation*}
where $(\theta)_{n \uparrow}: = \prod_{i=0}^{n-1} (\theta+i)$ is the {\em Pochhammer symbol}. This implies that
\begin{equation}
\label{pcts}
p_n = \widetilde{p}_n = \frac{(\theta)_{n \uparrow}}{n!} \quad \mbox{for all}~n>0.
\end{equation}
\quad By injecting \eqref{pcts} into \eqref{stationaryPi}, \eqref{trans1} and \eqref{trans2}, we get \eqref{disarcsine}, \eqref{eq5354} and \eqref{eq55}. The formula \eqref{eq56} is obtained by the following lemma.

\begin{lemma}
\label{lem11}
\begin{equation*}
P_{N}(0,N) = \frac{2(1-\theta)}{N+1} - \frac{(1 - 2 \theta)(2 \theta)_{N \uparrow}}{(N+1)(\theta)_{N \uparrow}}.
\end{equation*}
\end{lemma}
\begin{proof}
Note that $P_N(0,N) = 1 - \sum_{j=0}^{N-1} P_N(0,j)$. Thus , it suffices to show that
\begin{equation}
\label{magind}
\sum_{j=0}^{N-1} p_j p_{N-j} - \sum_{j=0}^{N-1} p_{j+1} p_{N-j} = \frac{1}{(N+1) !} \Bigg[ (N-2 \theta-1)(\theta)_{N \uparrow} + (1-2 \theta) (2 \theta)_{N \uparrow}\Bigg].
\end{equation}
Furthermore, for $|s|<1$,
\begin{equation*}
(1-s)^{-2\theta} = \left(\sum_{j=0}^{\infty}p_j s^j \right)^2 = \sum_{N=0}^{\infty} \left(\sum_{j=0}^N p_j p_{N-j} \right)s^j.
\end{equation*}
By identifying the coefficients on both sides, we get
\begin{equation*}
\sum_{j=0}^N p_j p_{N-j} = \frac{(2 \theta)_{N \uparrow}}{N !} \quad \mbox{and} \quad \sum_{j=0}^{N+1} p_jp_{N+1-j} = \frac{(2 \theta)_{N+1 \uparrow}}{(N+1) !}.
\end{equation*}
Therefore,
\begin{equation*}
\sum_{j=0}^{N-1} p_j p_{N-j} - \sum_{j=0}^{N-1} p_{j+1} p_{N-j} = \left[ \frac{(2 \theta)_{N \uparrow}}{N !} - \frac{(\theta)_{N \uparrow}}{N !}\right] - \left[\frac{(2 \theta)_{N+1 \uparrow}}{(N+1) !} - \frac{(\theta)_{N+1 \uparrow}}{(N+1) !}\right],
\end{equation*}
which leads to \eqref{magind}.
\end{proof}

\quad When $F$ is symmetric and continuous, the above results can be simplified. In this case, $\mathbb{P}(S_n \geq 0) = \mathbb{P}(S_n>0) = \frac{1}{2}$.
\begin{corollary}
Assume that $F$ is symmetric and continuous. Then the stationary distribution of the argmin chain $(A_N(n); n \geq 0)$ is given by
\begin{equation}
\Pi_N(k) = \binom{2k}{k} \binom{2N-2k}{N-k} 2^{-2N} \quad \mbox{for}~0 \leq k \leq N.
\end{equation}
In addition, the transition probabilities are
\begin{equation}
P_N(i,N) = \frac{1}{2(N+1-i)} \quad \mbox{and} \quad P_N(i,i-1) = \frac{2N +1 -2i}{2(N+1-i)} \quad \mbox{for}~0<i \leq N;
\end{equation}
\begin{equation}
P_N(0,j) = \frac{\binom{N}{j}^2}{2(j+1)\binom{2N}{2j}} \quad \mbox{for}~0 \leq j <N \quad \mbox{and} \quad P_N(0,N) = \frac{1}{N+1}.
\end{equation}
\end{corollary}
\subsection{Simple symmetric random walks} 
\label{s52}
In \cite[Chapter III.3]{Feller}, Feller found for a simple symmetric walk,
\begin{equation}
\label{ptil2}
\widetilde{p}_{2n} = \widetilde{p}_{2n+1} = \frac{(\frac{1}{2})_{n \uparrow}}{2 \cdot n!} \quad \mbox{for all}~n \geq 1,
\end{equation}
and 
\begin{equation}
\label{p2}
p_{2n-1} = p_{2n} = \frac{(\frac{1}{2})_{n \uparrow}}{ n!} \quad \mbox{for all}~n \geq 1.
\end{equation}

\quad By injecting \eqref{ptil2} and \eqref{p2} into \eqref{stationaryPi}, \eqref{trans1} and \eqref{trans2}, we get \eqref{eq513}, \eqref{eq514} and \eqref{eq515}. The formula \eqref{eq516} is obtained by the following lemma.
\begin{lemma}
\label{lem12}
\begin{equation*}
P_N(0,N) = \left\{ \begin{array}{ccl}
         \frac{1}{N+1} & \mbox{if}~ N ~\mbox{is odd}; \\ [8 pt]
         \frac{2}{N+2} & \mbox{if} ~N ~\mbox{is even}.
                \end{array}\right.
\end{equation*}
\end{lemma}
\begin{proof}
Note that $P_N(0,N) = 1 -\sum_{j=0}^{N-1} P_N(0,j)$. Thus, it suffices to show that
\begin{equation}
\label{magindbis}
\sum_{j=0}^{N-1} p_j p_{N-j} - \sum_{j=0}^{N-1} p_{j+1}p_{N-j} = 
\left\{ \begin{array}{ccl}
         \frac{N}{N+1}\widetilde{p}_N & \mbox{if}~ N ~\mbox{is odd}; \\ 
         \frac{N}{N+2}\widetilde{p}_N & \mbox{if} ~N ~\mbox{is even}.
                \end{array}\right.
\end{equation}
Furthermore, for $s<1$,
\begin{equation*}
\frac{1}{1-s} = \left(\sum_{j=0}^{\infty}p_j s^j \right) \left(\sum_{j=0}^{\infty}\widetilde{p}_j s^j \right) = \sum_{N=0}^{\infty} \left(\sum_{j=0}^N p_j \widetilde{p}_{N-j} \right)s^j.
\end{equation*}
By identifying the coefficients on both sides, we get
\begin{equation*}
\sum_{j=0}^N p_j \widetilde{p}_{N-j} = \sum_{j=0}^{N+1} p_j \widetilde{p}_{N+1-j} = 1.
\end{equation*}
Therefore,
\begin{align*}
\sum_{j=0}^{N-1} p_j p_{N-j} - \sum_{j=0}^{N-1} p_{j+1}p_{N-j} = (1-p_N) - (1-p_{N+1}-\widetilde{p}_{N+1}),
\end{align*}
which leads to \eqref{magindbis}.
\end{proof}
\section{The argmin process of L\'evy processes}
\label{s3}
\quad In this section, we consider the argmin process $(\alpha^X_t; t \geq 0)$ of a L\'evy process $(X_t; t \geq 0)$. 
According to the L\'evy-Khintchine formula, the characteristic exponent of $(X_t; t \geq 0)$ is given by
$$\Psi_X (\theta) : = ia \theta + \frac{\sigma^2}{2} \theta^2 + \int_{\mathbb{R}} (1-e^{i\theta x} + i \theta x 1_{\{|x|<1\}}) \Pi(dx),$$ 
where $a \in \mathbb{R}$, $\sigma \geq 0$, and  $\Pi(\cdot)$ is the L\'evy measure satisfying $\int_{\mathbb{R}} \min(1,x^2) \Pi(dx) < \infty$. The L\'evy process $X$ is a compound Poisson process if and only if 
$\sigma = 0$ and $\Pi(\mathbb{R}) < \infty$. 
In this case, the process $X$ has the following representation:
\begin{equation}
\label{CPP}
X_t = ct + \sum_{i=1}^{N_t} Y_i \quad \mbox{for all}~t>0,
\end{equation}
where 
$c = -a - \int_{|x|<1} x \Pi(dx)$,
$(N_t; t \geq 0)$ is a Poisson process with rate $\lambda$, and $(Y_i; i \geq 1)$ are independent and identically distributed random variables with cumulative distribution function $F$, independent of $N$ and satisfying $\lambda F(dx) = \Pi(dx)$.  
See Bertoin \cite{Bertoin} and Sato \cite{Sato} for further development on L\'evy processes.

\quad In Section \ref{s31}, we review Millar-Denisov's decomposition for L\'evy processes with continuous distribution. In Section \ref{s32}, we consider a path decomposition for compound Poisson processes. Finally in Section \ref{s33}, we explain how to adapt the arguments in Pitman and Tang \cite{PTargmin} to prove Theorem \ref{mainLevy}.
\subsection{Millar-Denisov's decomposition for L\'evy processes}
\label{s31}
For $A \in \mathcal{B}(\mathbb{R})$, let 
$T_A : = \inf\{t>0: X_t \in A\}$
be the hitting time of $A$ by $(X_t; t \geq 0)$. 
Recall that $0$ is regular for the set $A$ if $\mathbb{P}(T_A = 0) = 1$.

\quad Assume that $X$ is not a compound Poisson process with drift, which is equivalent to
\begin{itemize}
\item[(CD).]
For all $t>0$, $X_t$ has a continuous distribution; that is for all $x \in \mathbb{R}$,
$\mathbb{P}(X_t = x) = 0$.
\end{itemize}
See Sato \cite[Theorem 27.4]{Sato}.
According to Blumenthal's zero-one law, $0$ is regular for at least one of the half-lines $(-\infty,0)$ and $(0,\infty)$. There are three subcases: 
\begin{itemize}
\item[(RB).]
$0$ is regular for both half-lines $(-\infty,0)$ and $(0,\infty)$;
\item[(R$+$).]
$0$ is regular for the positive half-line $(0,\infty)$ but not for the negative half-line $(-\infty,0)$;
\item[(R$-$).]
$0$ is regular for the negative half-line $(-\infty,0)$ but not for the positive half-line $(0,\infty)$.
\end{itemize}

\quad Millar \cite{Millar78} proved that almost surely $(X_t; 0 \leq t \leq 1)$ achieves its minimum at a unique time $A \in [0,1]$, and
\begin{itemize}
\item
under the assumption (RB), $X_{A-} = X_A = \inf_{t \in [0,1]}X_t$ almost surely;  
\item
under the assumption (R$+$), $X_{A-} > X_A = \inf_{t \in [0,1]}X_t$ almost surely;
\item
under the assumption (R$-$), $X_A > X_{A-} = \inf_{t \in [0,1]}X_t$ almost surely.
\end{itemize}
The following result is a simple consequence of Millar \cite[Proposition 4.2]{Millar78}.

\begin{theorem} \cite{Millar78}
\label{M78}
Assume that $(X_t; 0 \leq t \leq 1)$ is not a compound Poisson process with drift. Let $A$ be the a.s. unique time such that
$
\inf_{t \in [0,1]} X_t = \min (X_{A-}, X_A).
$
Given $A$, the L\'evy path is decomposed into two conditionally independent pieces:
$$\left(X_{(A-t)-}-\inf_{u \in [0,1]}X_u; 0 \leq t \leq A\right) \quad \mbox{and} \quad \left(X_{A+t}-\inf_{u \in [0,1]}X_u; 0 \leq t \leq 1-A\right).$$
\end{theorem}

\quad In \cite{Millar78}, Millar provided the law of the post-$A$ process $\left(X_{A+t}-\inf_{t \in [0,1]}X_t; 0 \leq t \leq 1-A\right)$ but he did not mention the law of the pre-$A$ process $\left(X_{(A-t)-}-\inf_{t \in [0,1]}X_t; 0 \leq t \leq A\right)$. Relying on Chaumont-Doney's construction \cite{CD10} of L\'evy meanders, Uribe Bravo \cite{UB14} proved that if $(X_t; 0 \leq t \leq 1)$ is not a compound Poisson process with drift and satisfies the assumption (RB), then
\begin{itemize}
\item
$\left(X_{(A-t)-}-\inf_{u \in [0,1]}X_u; 0 \leq t \leq A\right)$ is a L\'evy meander of length $A$;
\item
$(X_{A+t}-\inf_{u \in [0,1]}X_u; 0 \leq t \leq 1-A)$ is a L\'evy meander of length $1-A$.
\end{itemize}
This result generalizes Denisov's decomposition to L\'evy processes with continuous distribution. See also Chaumont \cite{Chaumont97} for related results for stable L\'evy processes.
\subsection{A path decomposition for compound Poisson processes}
\label{s32}
Here we give a path decomposition for compound Poisson processes, which is left out in the literature. Since a compound Poisson process is a continuous-time random walk, our construction is based on Denisov's decomposition for random walks.

\quad Let $(X_t; 0 \leq t \leq 1)$ be a compound Poisson process defined by \eqref{CPP}. Let 
\begin{equation*}
A : = \sup\left\{0 \leq s \leq 1: X_{s} = \inf_{u \in [0,1]} X_{u}\right\}
\end{equation*}
be the last time at which $X$ achieves its minimum on $[0,1]$. Let $\{\xi_1, \cdots, \xi_N\}$ be the jumping positions of $(N_t; 0 \leq t \leq 1)$, with $N:=N_1$ and $0=:\xi_0 < \xi_1<\cdots <\xi_{N} < \xi_{N+1}:=1$.

\quad Given $\xi_1,\cdots,\xi_{N}$ and $A = \xi_k$ for some $k \leq N$, we distinguish two cases:
\vskip 6 pt
{\bf Case $1$.} $c \leq 0$. Then $\inf_{t \in [0,1]} X_t = X_{A-}$. Define $Y^c_i:=Y_i + c(\xi_{i+1} - \xi_i)$ for $1 \leq i \leq N$.
Let $\overset{\rightarrow}{Z}$ be distributed as $\overset{\rightarrow \quad \quad}{S^{Y^c}_{N-k+1}}(k-1)$ conditioned to stay positive, and 
$\overset{\leftarrow}{Z}$ be distributed as $\overset{\leftarrow \quad}{S^{Y^c}_{k-1}}(k-1)$ conditioned to stay non-negative. Define
$$
\overset{\leftarrow}{X}_t : =  \left\{ \begin{array}{cl}        
-ct  & \mbox{for}~0 \leq t < A - \xi_{k-1}, \\ 
\overset{\leftarrow}{Z}_1-c(t-A + \xi_{k-1}) & \mbox{for}~ A - \xi_{k-1} \leq  t < A - \xi_{k-2}, \\ 
\vdots & \vdots \\
\overset{\leftarrow}{Z}_{k-1} - c(t-A+\xi_1)  & \mbox{for}~A - \xi_1\leq t \leq A.
\end{array}\right.
$$
and
$$
\overset{\rightarrow}{X}_t:=  \left\{ \begin{array}{cl}        
\overset{\rightarrow}{Z}_1- c(\xi_{k+1} - A - t)  & \mbox{for}~0 \leq t < \xi_{k+1} - A, \\ 
\overset{\rightarrow}{Z}_2 - c(\xi_{k+2} - A - t)& \mbox{for}~ \xi_{k+1} - A \leq  t < \xi_{k+2} - A, \\ 
\vdots & \vdots \\
\overset{\rightarrow}{Z}_{N-k+1}-c(1-A-t)  & \mbox{for}~\xi_N - A \leq t \leq 1 - A.
\end{array}\right.
$$
{\bf Case $2$.} $c>0$. Then $\inf_{t \in [0,1]} X_t = X_{A}$. Define $Y^c_i:=Y_i + c(\xi_{i} - \xi_{i-1})$ for $1 \leq i \leq N$.
Let $\overset{\rightarrow}{Z}$ be distributed as $\overset{\rightarrow \quad}{S^{Y^c}_{N-k}}(k)$ conditioned to stay positive, and 
$\overset{\leftarrow}{Z}$ be distributed as $\overset{\leftarrow \quad}{S^{Y^c}_{k}}(k)$ conditioned to stay non-negative. Define
$$
\overset{\leftarrow}{X}_t : =  \left\{ \begin{array}{cl}        
\overset{\leftarrow}{Z}_1  + c(A - \xi_{k-1} - t)& \mbox{for}~0 \leq t < A - \xi_{k-1}, \\ 
\overset{\leftarrow}{Z}_2 + c(A - \xi_{k-2} - t) & \mbox{for}~ A - \xi_{k-1} \leq  t < A - \xi_{k-2}, \\ 
\vdots & \vdots \\
\overset{\leftarrow}{Z}_{k} + c(A - t)  & \mbox{for}~A - \xi_1\leq t \leq A.
\end{array}\right.
$$
and
$$
\overset{\rightarrow}{X}_t:=  \left\{ \begin{array}{cl}        
ct  & \mbox{for}~0 \leq t < \xi_{k+1} - A, \\ 
\overset{\rightarrow}{Z}_1 + c(t -\xi_{k+1} + A )& \mbox{for}~ \xi_{k+1} - A \leq  t < \xi_{k+2} - A, \\ 
\vdots & \vdots \\
\overset{\rightarrow}{Z}_{N-k} + c(t-\xi_N + A)  & \mbox{for}~\xi_N - A \leq t \leq 1 - A.
\end{array}\right.
$$
By Theorem \ref{DenisovRW}, 
the path of $X$ is decomposed into two conditionally independent pieces:
\begin{equation}
\left(X_{(A-t)-}-\inf_{u \in [0,1]}X_u; 0 \leq t \leq A\right) \stackrel{(d)}{=} (\overset{\leftarrow}{X}_t; 0 \leq t \leq A),
\end{equation}
and
\begin{equation}
\left(X_{A+t}-\inf_{u \in [0,1]}X_u; 0 \leq t \leq 1-A \right) \stackrel{(d)}{=}  (\overset{\rightarrow}{X}_t; 0 \leq t \leq 1-A).
\end{equation}
\subsection{The Markov property of the argmin process}
\label{s33}
Combining the results in the last two sections, we have the following corollary which is a slight extension to Theorem \ref{M78}.
\begin{corollary} 
\label{haha}
Let $(X_t; 0 \leq t \leq 1)$ be a real-valued L\'evy process. Let 
$$A : = \sup\left\{0 \leq s \leq 1: X_{s} = \inf_{u \in [0,1]} X_{u}\right\}$$
be the last time at which $X$ achieves its minimum on $[0,1]$. Given $A$, the path of $X$ is decomposed into two conditionally independent pieces:
$$\left(X_{(A-t)-}-\inf_{u \in [0,1]}X_u; 0 \leq t \leq A\right) \quad \mbox{and} \quad \left(X_{A+t}-\inf_{u \in [0,1]}X_u; 0 \leq t \leq 1-A\right).$$
\end{corollary}
With Corollary \ref{haha}, it is easy to adapt the argument of Pitman and Tang \cite[Proposition 3.4]{PTargmin} to prove that $(\alpha^X_t; t \geq 0)$ is a time-homogeneous Markov process. Here the detail is omitted.

\quad Now we turn to the stable L\'evy process. Let $(X_t; t \geq 0)$ be a stable L\'evy process with parameters $(\alpha,\beta)$, and neither $X$ nor $-X$ is a subordinator. It is well known that $0$ is regular for the reflected process $X- \underline{X}$. So It\^o's excursion theory can be applied to the process $X- \underline{X}$, see Sharpe \cite{Sharpe} for background on excursion theory of Markov processes. 

\quad Let ${\bf n}(d\epsilon)$ be the It\^o measure of excursions of $X-\underline{X}$ away from $0$. Monrad and Silverstein \cite{MS79} computed the law of lifetime $\zeta$ of excursions under ${\bf n}$:
\begin{equation}
{\bf n}(\zeta > t) = c \frac{t^{\rho-1}}{\Gamma (\rho)}  \quad \mbox{and} \quad {\bf n}(\zeta \in dt) = c(1-\rho) \frac{t^{\rho-2}}{\Gamma(\rho)}
\end{equation}
for some contant $c>0$. Following the argument of Pitman and Tang \cite[Remark 3.9]{PTargmin}, we have:
\begin{proposition}
Let $(X_t; t \geq 0)$ be a stable L\'evy process with parameters $(\alpha,\beta)$, and neither $X$ nor $-X$ is a subordinator. Then the jump rate of the argmin process $\alpha^X$ per unit time from $x \in (0,1)$ to $1$ is given by
\begin{equation}
\mu^{\uparrow 1}(x) = \frac{1 - \rho}{1-x} \quad \mbox{for}~0 < x <1.
\end{equation}
\end{proposition}
\quad Finally, by doing similar calculations as in Pitman and Tang \cite[Section 3.4]{PTargmin}, we obtain the Feller transition semigroup \eqref{Qtrans} for $X$. 

\bibliographystyle{plain}
\bibliography{unique}
\end{document}